\documentclass[11pt,a4paper]{amsart}
\usepackage{amsfonts,amsmath}
\usepackage{latexsym}
\usepackage{times}
\allowdisplaybreaks

\usepackage{amsmath,amssymb,amsfonts,amsthm}
\usepackage[english]{babel}

\newtheorem{proposition}{Proposition}[section]
\newtheorem{theorem}[proposition]{Theorem}
\newtheorem{corollary}[proposition]{Corollary}
\newtheorem{lemma}[proposition]{Lemma}
\newtheorem{remark}[proposition]{Remark}

\newcommand{\nc}{\newcommand}
\nc{\R}{\mathbb{R}}
\nc{\K}{\mathcal{K}}
\nc{\NK}{\mathcal{N}(X)}
\nc{\LK}{\mbox{\bf [}}
\nc{\RK}{\mbox{\bf ]}}
\nc{\Rd}{\mathbb{R}^{d}}
\nc{\I}{{\bf 1}}
\nc{\Z}{{\mathbb Z}}
\nc{\meares}{\,\llcorner\,}
\nc{\mbK}{\mathbb{K}}
\nc{\mbX}{\mathbb{K}}
\nc{\nor}{{\rm nor}\,}
\nc{\Nor}{{\rm Nor}}
\nc{\Tan}{{\rm Tan}}
\nc{\cone}{{\rm cone}}
\nc{\SOd}{{\rm SO}(d)}
\nc{\reach}{\operatorname{reach}}
\nc{\bC}{{\bf C}}
\nc{\dist}{\operatorname{dist}}
\nc{\Ha}{{\mathcal H}}

\begin{document}

\author{Daniel Hug}
\thanks{The research has been supported by the DFG project HU
1874/4-2 (D. Hug) and the GA\v CR project P201/15-08218S (J. Rataj)} 
\address{Karls\-ruhe Institute of Technology (KIT), Department of Mathematics, 
D-76128 Karls\-ruhe, Germany}
\email{daniel.hug@kit.edu}
\urladdr{http://www.math.kit.edu/$\sim$hug/}
\author{Jan Rataj}
\address{Charles University, Faculty of Mathematics and Physics, 
Sokolovska 83, 186 75 Praha 8, 
Czech Republic}
\email{rataj@karlin.mff.cuni.cz}
\urladdr{http://www.karlin.mff.cuni.cz/$\sim$rataj/index\underline{ }en.html}
\title{Mixed curvature measures of translative integral geometry}
\date{\today}

\begin{abstract}
The curvature measures of a set $X$ with singularities are measures concentrated 
on the normal bundle of $X$, which describe the local geometry of the set $X$. 
For given finitely many convex bodies or, more generally, sets with positive reach, the translative integral formula for curvature measures 
 relates the integral mean of the curvature measures of the intersections of the given sets, one fixed and the others translated, to the mixed curvature measures of the given  sets. In the case of two sets of positive reach,  
a representation of these mixed measures in terms of generalized curvatures, defined on the normal bundles of the sets, is known. For more than two sets, a description of mixed curvature measures in terms of rectifiable currents has been derived previously. Here we provide a representation 
of mixed curvatures measures of sets with positive reach based on generalized curvatures. 
 The special case of convex polyhedra is treated in detail.
\end{abstract}

\maketitle

\section{Introduction}
The {\it reach} of a set $X\subset\R^d$, denoted $\reach X$, is the infimum of all $r\geq 0$ such that for each point $z\in\R^d$ with $\dist(X,z)\leq r$ there is a unique nearest point $\Pi_X(z)$ in $X$. Sets with positive reach were studied first by Federer \cite{F59} who showed that they satisfy a local Steiner formula, that is, for any $0<r<\reach X$ and any Borel set $ B\subset\R^d$,
\begin{equation}  \label{SF}
\Ha^d(X_r\cap\Pi_X^{-1}(B))=\sum_{k=0}^d \kappa_{d-k}r^{d-k}\bC_k(X,B),
\end{equation}
where $X_r:=\{z\in\R^d:\, \dist(z,X)\leq r\}$ and $\kappa_j:=\pi^{\frac j2}/\Gamma(1+\frac j2)$. The coefficients $\bC_k(X,\cdot)$ are signed Radon measures, called {\it curvature measures} of order $k$ of $X$ if $0\leq k\leq d-1$, and $\bC_d(A,\cdot)=\Ha^d(X\cap\cdot)$. The curvature measures possess the usual properties of curvature measures of sets with $C^2$ smooth boundaries and of convex sets, in particular, they satisfy the Gauss-Bonnet formula and the Principal Kinematic Formula (see \cite{F59}).

The main difference to the smooth case is that the Gauss map is not defined uniquely on the boundary of a set $X$ with positive reach. Therefore, the {\it unit normal bundle} 
$$\nor X:=\{(x,u)\in\R^d\times S^{d-1}:\, x\in X,\, u\in\Nor(X,x)\}$$
is used instead (here $\Nor(X,x)$ is the normal cone of $X$ at $x\in X$, defined as the dual convex cone to the tangent cone $\Tan(X,x)$), and the role of the Gauss map from the smooth case is played by the projection $(x,u)\mapsto u$ to the second component. Thus, 
in generalization of the curvature measures $\bC_k(X,\cdot)$ on $\R^d$, 
it is convenient to consider curvature measures as measures  on $\R^d\times S^{d-1}$ which are supported by the unit normal bundle of $X$. Such measures are determined by the refined local Steiner formula which states that, for any $0<r<\reach X$ and any bounded Borel set $ A \subset\R^d\times S^{d-1}$, 
\begin{equation}  \label{SF2}
\Ha^d((X_r\setminus X)\cap\xi_X^{-1}(A))=\sum_{k=0}^{d-1} \kappa_{d-k}r^{d-k}C_k(X,A ),
\end{equation}
where $\xi_X:z\mapsto (\Pi_X(z),\frac{z-\Pi_X(z)}{\|z-\Pi_X(z)\|})$, $z\in X_r\setminus X$. The coefficients $C_k(X,\cdot)$ are 
 signed Radon  measures  on $\R^d\times S^{d-1}$, their first component projections agree with the curvature measures from \eqref{SF} and they are called {\em generalized curvature measures} \cite{Z86}, support measures \cite{SW08} or curvature-direction measures. In the following, we shall also use the short name curvature measures for the measures in \eqref{SF2}.

One starting point of the present work are 
kinematic formulas of integral geometry for sets $X,Y\subset\R^d$ of positive reach. Let ${\rm G}_d$ denote the Euclidean motion group of $\R^d$ and let $\mu_d$ denote the suitably normalized Haar measure on ${\rm G}_d$. For bounded Borel sets $\alpha,\beta\subset\R^d$, the 
principal kinematic formula for curvature measures states that 
$$
\int_{{\rm G}_d}\bC_k(X\cap g Y,\alpha\cap g\beta)\, \mu_d(dg)=\sum_{\substack{0\le i,j\le d\\i+j=d+k}}c(d,i,j)\bC_i(X,\alpha)\bC_j(Y,\beta),
$$
where  $c(d,i,j)$ are explicitly known constants (see 
\cite{RaZ95}, \cite{SW08}). 
 In many applications in stochastic geometry it is, however, necessary to consider integration with respect to translations only. In particular, this is crucial for the investigation of stationary random sets which are not isotropic (see \cite{SW08}). The basic formula of {\it translative} integral geometry thus deals with the integrals 
$$\int_{\R^d} \bC_k(X\cap(Y+z),\alpha\cap(\beta+z))\, dz,$$
which are expressed as a sum of mixed curvature measures depending on both sets $X$ and $Y$. More generally, using the generalized curvature measures and an arbitrary measurable function $h:\R^{2d}\times S^{d-1}\to [0,\infty]$ with compact support (allowing to include 
directional information), we are interested in the translative integrals 
$$
\int_{\R^d}\int h(x,x-z,u)\, C_k(X\cap (Y+z),d(x,u))\, dz,
$$
which  again can be expressed in terms of integrals of mixed curvatures measures of $X$ and $Y$. 
The iterated version of such a relation works with a finite number $q$ of sets, $q-1$ of them being shifted independently. 
For $q\ge 2$ and given subsets $X_1,\ldots,X_q$ of $\Rd$ with positive reach, the iterated translative integral formula involves the {\it mixed curvature measures}
\[C_{r_1,\ldots,r_q}(X_1,\ldots,X_q;\cdot)\,,\]
for  $r_1,\ldots,r_q\in\{0,\ldots,d\}$ with $r_1+\ldots+r_q\ge(q-1)d$, which are signed Borel measures on $\R^{qd}\times S^{d-1}$, and reads
\begin{eqnarray}  \label{TIF}
&&{\int_{\R^d}\ldots\int_{\R^d}\int  
h(x,x-z_2,\ldots,x-z_q,u)\, C_k(\underline{X}(\underline{z}),d(x,u))}\,dz_q\ldots dz_2 \nonumber\\
&&\qquad =\sum_{\substack{0\le r_1,\ldots,r_q\le d\\ r_1+\ldots+r_q=(q-1)d+k}}
\int h(x_1,\ldots,x_q,u)\,C_{r_1,\ldots,r_q}(X_1,\ldots,X_q;d(x_1,\ldots,x_q,u)),
\end{eqnarray}
where  $k\in\{0,\ldots,d-1\}$, $\underline{X}(\underline{z}):=X_1\cap(X_2+z_2)\cap\ldots\cap(X_q+z_q)$, and $h:\R^{qd}\times S^{d-1}\to [0,\infty]$ is an arbitrary Borel measurable function with compact support. 

This iterated integral formula was first proved in the setting of convex geometry  
by Schneider and Weil \cite{SchneiderandWeil1986} for $q=2$ and by
Weil \cite{Weil1990} for $q\ge 2$ 
in a less general form, namely for a function $h$ which is independent of the direction vector $u$. Subsequently, formula 
\eqref{TIF} was established in \cite{RaZ95} for $q=2$, and in \cite{Rataj1997} for general
$q$, in the setting of sets with positive reach. An extension to
relative curvature measures, that is, curvature measures defined with respect to a non-Euclidean metric, has been obtained in \cite[Section~3]{HugHab}.

For the mixed curvature measures of arbitrary sets with positive reach and $q\ge 3$,  
up to now  only a 
representation was available  which involves the notion of a rectifiable current (see \cite{Rataj1997}). In the special case 
of  mixed curvature measures of two  sets of positive reach (that is, for $q=2$) an integral
representation based on generalized curvature functions, defined on the normal bundles of the sets,  
has already been proved in \cite{RaZ95,RaZ01}, while the case of convex bodies and general $q$ is covered in \cite[Section~4]{HugHab}. In the present paper, we extend all these results by treating the case of  
 a finite sequence of sets with positive reach. For convex polyhedra we obtain a simple description of the 
mixed curvature measure which has an intuitive geometric interpretation (see also \cite[Section~4]{HugHab}) and extends 
the  important special case considered in \cite{WeilPreprint}. 

In Section~2 we introduce the notions and notation used and provide two auxiliary results, one from multilinear algebra and the other from  measure theory. In Section~3 we formulate our main result (Theorem~\ref{MiCurv}) and provide sufficient conditions for the validity of its assumption. We also deal with some important particular cases such as that of convex polyhedra. The last section (Section~4) contains the proof of the main result.

\section{Preliminaries}

The basic setting for this paper will be the $d$-dimensional Euclidean space $\Rd$,
$d\ge 2$, with scalar product $x\cdot y$ and norm $|x|=\sqrt{x\cdot x}$,
$x,y\in\Rd$. 
The same notation will be adopted in any Euclidean space which is treated,
independent of its dimension.  In particular, we 
shall investigate cartesian products such as $\Rd\times\ldots\times\Rd$, with $k$ 
factors, for which we also write  $\R^{k d}$. In this case, we endow
each factor with the same scalar product, and the cartesian product will carry 
the natural scalar product which is derived from its components by summation. 
Let $\mathcal{H}^s$, for $s\ge 0$, denote the $s$-dimensional Hausdorff measure. 
The Euclidean spaces where Hausdorff measures will be considered, will always 
be clear from the context. We write 
$\omega_n:=2\pi^{n/2}/\Gamma\left(n/2\right)$ for the $(n-1)$-dimensional Hausdorff measure 
of the $(n-1)$-dimensional unit sphere $S^{n-1}$ in $\R^n$. 

We shall use the standard notation of multilinear algebra as introduced
in \cite{F69}. In particular, for $k\in\{0,\ldots,d\}$ we denote by  $\bigwedge_kV$ and $\bigwedge^kV$ 
the spaces of $k$-vectors and $k$-covectors, recpectively, of a vector
space $V$, and $\langle\alpha ,\phi\rangle$ stands for the bilinear
pairing, where $\alpha\in\bigwedge_kV$ and $\phi\in\bigwedge^kV$. We denote by
$\Omega^d=e_1'\wedge\cdots\wedge e_d'$ the volume $d$-form in $\Rd$, where 
$\{ e_1',\ldots ,e_d'\}$ is  basis which is dual to the canonical
orthonormal basis $\{ e_1,\ldots ,e_d\}$ of $\Rd$.
The scalar product in $\Rd$ induces a natural linear isomorphism
$v\mapsto v'$ from $\Rd$ to the dual space $\bigwedge^1\Rd$ which
in turn induces a natural linear isomorphism $\alpha\mapsto\alpha '$
from $\bigwedge_k\Rd$ to its dual $\bigwedge^k\Rd$. By means of this
correspondence, the mapping $\alpha\mapsto *\alpha$ from
$\bigwedge_k\Rd$ to $\bigwedge_{d-k}\Rd$ is defined (cf.\ \cite{F69}) by 
$$*\alpha =(e_1\wedge\cdots\wedge e_d)\meares\alpha ',$$
where $\cdot\meares\cdot$ denotes the standard inner multiplication
(see \cite[\S 1.5.1 and \S 1.7.8]{F69}). It follows from the definition that
\begin{equation}    \label{star}
\langle\alpha\wedge *\alpha ,\Omega^d\rangle =|\alpha|^2.
\end{equation}

Let $p$ be a natural number and let $p$ multivectors 
$\alpha_1,\ldots,\alpha_p$ in $\Rd$ be given such that the sum of their multiplicities
equals $(p-1)d$. Then we define the {\it $p$-product} of
$\alpha_1,\ldots ,\alpha_p$ as
\begin{equation}\label{star2}
[\alpha_1,\ldots ,\alpha_p]=\langle (*\alpha_1)\wedge\cdots\wedge
(*\alpha_p),\Omega^d\rangle.
\end{equation}
Note that this definition is consistent with that given in
\cite{Rataj1997}. Moreover, if $\alpha_i$ is a unit simple
multivector and $L_i$ is the linear subspace corresponding to
$\alpha_i$, for $i=1,\ldots ,p$, then the $p$-product $[\alpha_1,\ldots ,\alpha_p]$
coincides, up to sign, with the function $[L_1,\ldots ,L_p]$ defined in
\cite{Weil1990} (see also \cite[\S 14.1]{SW08}). 

Let $q\geq 1$, $d\ge 2$ and  
$r_1,\ldots,r_q\in\{0,\ldots,d\}$ be given with 
$$(q-1)d \leq r_1+\ldots+r_q\leq qd-1. $$
We set $R_1:=r_1$, $R_2:=r_1+r_2$, \ldots ,
$R_q:=r_1+\ldots+r_q$, $r_{q+1}:=qd-1-R_q$ and 
$k:=r_1+\ldots+r_q-(q-1)d\in\{0,\ldots,d-1\}$; 
hence $r_{q+1}=d-1-k$. Let $\text{Sh}(r_1,\ldots,r_{q+1})$ denote the set of 
all permutations of $\{1,\ldots,qd-1\}$ which are increasing on each of the 
sets $\{1,\ldots,R_1\}$, $\{R_1+1,\ldots,R_2\}$, \ldots,
$\{R_q+1,\ldots,qd-1\}$. 

We write
$\varphi_{r_1,\ldots,r_q}\in\mathcal{D}^{qd-1}(\R^{(q+1)d})$ for the 
differential form which is defined by
\begin{align*}
&\left\langle\bigwedge_{i=1}^{qd-1}\left(a_i^1,\ldots,a_i^{q+1}\right),
\varphi_
{r_1,\ldots,r_q}(x_1,\ldots,x_q,u)\right\rangle\\
&\qquad=\frac{1}{\omega_{d-k}}(-1)^{c_1(d,r_1,\ldots,r_q)}\sum_{\sigma\in\text{Sh}(r
_1,
\ldots,r_{q+1})}\text{sgn}(\sigma)\\
&\qquad\qquad\times\left[\bigwedge_{i=1}^{R_1}a_{\sigma(i)}^1,\bigwedge_{i=R_1
+
1}^{R_2}
a_{\sigma(i)}^2,\ldots,\bigwedge_{i=R_{q-1}+1}^{R_q}a^q_{\sigma(i)},
\bigwedge_{i=R_q+1}^{qd-1}a_{\sigma(i)}^{q+1}\wedge u\right]\,,
\end{align*}
where $a_j^i\in\R^d$, for $i\in\{1,\ldots,q+1\}$ and $j\in\{1,\ldots,qd-1\}$,  
is arbitrarily chosen and
$$c_1(d,r_1,\ldots,r_q)=d\sum_{i=1}^qr_i+d\sum_{i=1}^qir_i+\sum_{1\leq i<j\leq q}r_ir_j.$$
Since $\varphi_{r_1,\ldots
,r_q}(x_1,\ldots ,x_q,u)$ depends only on the last vector component, we
shall write briefly $\varphi_{r_1,\ldots ,r_q}(u)$. 

In particular, for $q=1$, $r_1=:r\in\{0,\ldots,d-1\}$, $r_2=d-1-r_1$ and $k=r_1$, we have $c_1(d,r)=1$ and $\varphi_r$ is the $k$-th Lipschitz-Killing curvature form on $\R^{2d}$ involved in the definition of the $k$-th curvature measure (see \cite{Z86,RaZ95} for 
an alternative representation of this differential form). 
The sign determined by $c_1(d,r_1,\ldots,r_q)$ differs from that given in \cite{Rataj1997}, see the proof of Lemma \ref{Le2.1} below for a correction of the last step of the proof of \cite[Lemma 2]{Rataj1997}. 

Let $G_i,\pi$ be the projections defined on $(\R^d)^{q+1}$ by
$$G_i(x_1,\ldots,x_q,u)=x_1-x_i,\quad \pi (x_1,\ldots,x_q,u)=(x_1,u),$$
$i=2,\ldots,q$.

\begin{lemma}[{\cite[Lemma~2]{Rataj1997}}]\label{Le2.1}
For any $q\geq 2$ and $0\leq k\leq d-1$, we have
$$G_2^\#\Omega^d\wedge\cdots\wedge G_q^\#\Omega^d\wedge\pi^\#\varphi_k=
\sum_{\substack{0\leq r_1,\ldots,r_q\leq d\\r_1+\cdots +r_q=(q-1)d+k}}\varphi_{r_1,\ldots,r_q}.$$
\end{lemma}

\begin{proof}
This result was shown in \cite[Lemma~2]{Rataj1997}. In the last but one line of the proof, the sign  was still correct and given by  
$(-1)^{c_1}$ with 
$$c_1=(k-1)(q-1)d+\sum_{i=2}^q (d-r_i)(k_i-(i-2)d-1),$$
with $k_i=R_i-(i-1)d$. Using the symbol $m\sim n$ whenever two integers $m,n$ differ by an even number, we have
\begin{eqnarray*}
c_1(d,r_1,\ldots,r_q)&\sim&(q-1)dR_q+\sum_{i=1}^q(d-r_i)(R_i-d-1)\\
&\sim&(q-1)dR_q+d\sum_{i=1}^qR_i+\sum_{i=1}^qr_iR_i+(d-1)\sum_{i=1}^qr_i\\
&\sim&(q-1)dR_q+d\sum_{i=1}^q(q+1-i)r_i+\sum_{1\le i\leq j\le q}r_ir_j+(d-1)\sum_{i=1}^qr_i\\
&\sim&((q-1)d+d(q+1)+1+(d-1))\sum_{i=1}^qr_i+d\sum_{i=1}^qir_i+\sum_{1\le i<j\le q}r_ir_j\\
&\sim&d\sum_{i=1}^qr_i+d\sum_{i=1}^qir_i+\sum_{1\leq i<j\leq q}r_ir_j,
\end{eqnarray*}
which agrees with the value given in the definition above.
\end{proof}

Let $X\subset\Rd$ have positive reach, and let
$\nor X$ be its unit normal bundle, as defined in the introduction (cf.~\cite{F59}). 
Then $\nor X$ is locally $(d-1)$-rectifiable, 
and for ${\mathcal H}^{d-1}$-almost all $(x,u)\in\nor X$, the 
tangent cone of $\nor X$ at $(x,u)$ is the linear subspace spanned by the
vectors
\begin{equation}   \label{tannor}
\frac 1{\sqrt{1+k_i(x,u)^2}}\big( a_i(x,u),k_i(x,u)a_i(x,u)\big) ,\quad
i=1,\ldots ,d-1,
\end{equation}
where $k_1(x,u),\ldots k_{d-1}(x,u)\in (-\infty,\infty]$ are the (generalized) principal
curvatures and where $a_1(x,u),\ldots ,a_{d-1}(x,u)$ are the corresponding
principal directions at $(x,u)$ (cf.\ \cite{Z86}). In the case of infinite principal curvatures, 
we use the conventions $\frac 1{\sqrt{1+\infty^2}}=0$ and $\frac{\infty}{\sqrt{1+\infty^2}}=1$. The unit normal
bundle is oriented by a unit simple $(d-1)$-vector field $a_X(x,u)$
which can be given as the wedge product of the vectors
from \eqref{tannor} ordered in such a way that
$$\langle a_1(x,u)\wedge\cdots\wedge a_{d-1}(x,u)\wedge
u,\Omega^d\rangle =1.$$
Then the normal cycle of $X$ is the integer rectifiable current
$$N_X=({\mathcal H}^{d-1}\meares\nor X)\wedge a_X$$
and the $k$-th curvature measure of $X$, for $k\in\{0,\ldots ,d-1$\}, can be
represented as
$$C_k(X;A)=(N_X\meares\I_A)(\varphi_k),$$
where $A$ is a bounded Borel subset of $\Rd\times S^{d-1}$.

\subsection{Mixed curvature measures and the translative integral formula}

Let $q,d\ge 2$, and let $X_1,\ldots,X_q\subset\Rd$ be sets with positive
reach. For unit vectors
$u_1,\ldots,u_q\in S^{d-1}$ we set
\[\text{cone}\{u_1,\ldots,u_q\}:=\left\{\sum_{i=1}^q\lambda_iu_i:
\lambda_i\geq 0\text{ for }i=1,\ldots,q, \sum_{i=1}^q\lambda_i^2>0\right\}.\]
Note that $\text{cone}\{u_1,\ldots,u_q\}$ contains a line if and only if it contains the origin, otherwise it is a proper convex cone. Next we introduce the {\em
joint unit normal bundle}
\begin{eqnarray*}
\nor(X_1,\ldots,X_q)&:=&\{(x_1,\ldots,x_q,u)\in \R^{qd}\times S^{d-1}:\,
u\in \text{cone}\{u_1,\ldots,u_q\}\text{ for some 
}\\&&\,\;\,(x_i,u_i)\in\nor X_i,\, i=1,\ldots,q, \, o\not\in \text{cone}\{u_1,\ldots,u_q\}\}\,;
\end{eqnarray*}
compare \cite{Rataj1997}. Note that the open cone was used in
\cite{Rataj1997}. However, in order that \cite[Lemma~3]{Rataj1997} and
further results hold, the definition of the cone given here should be
applied. Further, we define the Borel sets
\[R^c:=\{(x_1,u_1,\ldots,x_q,u_q)\in(\R^d\times S^{d-1})^q:o\notin 
\text{cone}\{u_1,\ldots,u_q\}\}\, ,\]
$$\underline{\NK}:=(\nor X_1\times\ldots\times\nor X_q)\cap R^c$$
and
$$S^{q-1}_+=\{ (t_1,\ldots ,t_q)\in S^{q-1}:\, t_i\geq 0\mbox{ for }i=1,\ldots,q\}
.$$
The map
\[T:\underline{\NK}\times S^{q-1}_+\to \nor(X_1,\ldots,X_q)\]
is defined by
\[T(x_1,u_1,\ldots,x_q,u_q,t):=\left(x_1,\ldots,x_q,\frac{\sum_{i=1}^qt_iu_i}{
|
\sum_{i=1}^qt_iu_i
|}\right)\,.\]
It is easy to see that $T$ is well-defined, locally Lipschitz and onto.
Although $T$ is
not injective, the following lemma (proved for the case $q=2$ in
\cite{Zaehle1999}) is sufficient for our purposes.

\begin{lemma}\label{LemmaRataj}
For $\mathcal{H}^{qd-1}$-almost all elements of $\text{\rm im}(T)$, the 
pre-image 
under $T$ is a single point.
\end{lemma}

\begin{proof} We write
\[\triangle(q-1):=\left\{(t_1,\ldots,t_q)\in [0,1]^q:\sum_{i=1}^qt_i=1\right\}
\]
for the $(q-1)$-dimensional simplex embedded in $\R^q$. Clearly, to prove
the lemma  it is sufficient to show that the map
\begin{eqnarray*}
G&:&\underline{\NK}\times\triangle(q-1)\to\R^{qd}\times S^{d-1}\,,\\
&&(x_1,u_1,\ldots,x_q,u_q,t_1,\ldots,t_q)\mapsto\left(x_2-x_1,
\ldots,x_q-x_1,x_1,
\frac{\sum_{i=1}^q t_iu_i}{|\sum_{i=1}^qt_iu_i|}\right)\,,
\end{eqnarray*}
has a unique pre-image for $\mathcal{H}^{qd-1}$-almost all elements of 
$\text{im}(G)$. Excluding a set of $\mathcal{H}^{qd-1}$ measure zero from $\text{\rm im}(T)$, 
we see that it is sufficient to consider the
restriction $\tilde{G}$ of $G$ to the subset 
$\underline{\NK}\times\tilde{\triangle}(q-1)$ with
$\tilde{\triangle}(q-1)=\triangle(q-1)\setminus\{ (0,\ldots ,0,1)\}$. 

For the proof we proceed by induction. The case $q=2$ has been
established in \cite{Zaehle1999}. Now we assume that the assertion has already 
been proved for $q-1$ convex bodies. Set
\[\bar{R}^c:=\{(y_1,\ldots,y_{q-1},u,y,v)\in\R^{(q-1)d}\times 
S^{d-1}\times\R^d\times S^{d-1}:o\notin\text{cone}\{u,v\}\}\,.\]
To establish the assertion for $q$ sets $X_1,\ldots,X_q$ with positive
reach, $q\ge 3$, we introduce the maps
\[\varphi_{q}:\R^{qd}\times\R^d\to\R^{qd}\times\R^d\,,\quad 
(x_1,\ldots,x_q,u)\mapsto (x_2-x_1,\ldots,x_{q}-x_1,x_1,u)\,,\]
\begin{eqnarray*}
G_2&:&\underline{\NK}\times\tilde{\triangle}(q-1)\\
&&\qquad\qquad\to([\varphi_{q-1}(\nor(X_1,\ldots,X_{q-1}))
\times\nor X_q]\cap \bar{R}^c)\times(0,\infty)\,,\\
&&(x_1,u_1,\ldots,x_q,u_q,t_1,\ldots,t_q)\\
&&\qquad\qquad\mapsto\left(x_2-x_1,\ldots,x_{q-1}
-x_1,x_1,
\frac{\sum_{i=1}^{q-1} 
t_iu_i}{|\sum_{i=1}^{q-1}t_iu_i|},x_q,u_q,\frac{t_q}{|\sum_{i=1}^{q-1}t_iu_i|}
\right)\,,
\end{eqnarray*}
and 
\begin{eqnarray*}
G_1&:&([\varphi_{q-1}(\nor (X_1,\ldots,X_{q-1}))
\times\nor X_q]\cap \bar{R}^c)\times [0,\infty)
\to\varphi_{q}(\nor (X_1,\ldots,X_q))\,,\\
&&(z_2,\ldots,z_{q-1},x_1,v,x_q,u_q,s)
\mapsto\left(z_2,\ldots,z_{q-1},x_q-x_1,x_1,
\frac{v+su_q}{|v+su_q|}
\right)\,;
\end{eqnarray*}
hence, $\tilde{G}=G_1\circ G_2$. By the inductive hypothesis and since $(q-1)d-1+d=qd-1$, it follows that, for
$\mathcal{H}^{qd-1}$-almost all elements of $\text{im}(G_2)$, the map $G_2$ 
has 
a unique pre-image. Thus, since $G_1$ is locally Lipschitz, the image under 
$G_1$ of the set of all elements of $\text{im}(G_2)$ for which the pre-image 
under $G_2$ is not uniquely determined has $(qd-1)$-dimensional Hausdorff 
measure zero.

Furthermore, for $\mathcal{H}^{qd-1}$-almost all 
\[(z_2,\ldots,z_{q-1},x_1,v,x_q,u_q,s)\in  
([\varphi_{q-1}(\nor(X_1,\ldots,X_{q-1}))
\times\nor(X_q)]\cap \bar{R}^c)\times(0,\infty)\]
we have
\[(x_1,v)\in\nor(X_1\cap(X_2-z_2)\cap\ldots\cap(X_{q-1}-z_{q-1}))\,,\]
and therefore the result in \cite{Zaehle1999} shows that $\mathcal{H}^{qd-1}$-almost 
all elements of $G_1(\text{im}(G_2))$  have a unique pre-image under $G_1$. 
In 
fact, here we use that
\[\frac{v+su_q}{|v+su_q|}=\frac{\frac{1}{s+1}v+\frac{s}{s+1}u_q}
{|\frac{1}{s+1}v+\frac{s}{s+1}u_q|}\]
and $[0,\infty)\to [0,1)$, $s\mapsto(1+s)^{-1}s$, is locally bi-Lipschitz. Thus
the assertion follows.
\end{proof}

We recall now the description of the mixed curvature measures from
\cite{Rataj1997}. Since $T$ is locally Lipschitz, $\nor (X_1,\ldots
,X_q)$ is countably $(qd-1)$-rectifiable. We equip $\nor (X_1,\ldots
,X_q)$ with the orientation given by the unit simple tangent $(qd-1)$ vector field
$a_{X_1,\ldots ,X_q}$ associated with $\nor (X_1,\ldots
,X_q)$ fulfilling
\begin{equation}\label{possmall}
\langle a_{X_1,\ldots ,X_q},\psi_\varepsilon (u)\rangle >0
\end{equation}
for sufficiently small $\varepsilon >0$, where
$$\psi_\varepsilon (u)=\sum_{\substack{0\leq r_1,\ldots ,r_q\leq
d\\r_1+\cdots +r_q\geq (q-1)d}}\varepsilon^{qd-1-r_1-\cdots -r_q}
\varphi_{r_1,\ldots ,r_q}(u).$$
It follows from the proof of Theorem \ref{MiCurv} below that condition \eqref{possmall} 
is satisfied if $\varepsilon>0$ is small enough. 

\medskip

Let $0\leq r_1,\ldots ,r_q\leq d-1$, $r_1+\cdots +r_q\geq (q-1)d$, be
integers. The {\it mixed curvature measure} of $X_1,\ldots ,X_q$ of order
$r_1,\ldots ,r_q$ is a signed Radon measure defined by
\begin{equation} \label{mixed-c-m}
C_{r_1,\ldots,r_q}(X_1,\ldots,X_q;A):=\left[\left(\mathcal{H}^{qd-1}\meares
\nor(X_1,\ldots,X_q)\right)\wedge
a_{X_1,\ldots,X_q}\right](\I_A\varphi_{r_1,\ldots,r_q}) ,
\end{equation}
where $A\subset\R^{qd}\times S^{d-1}$ is a Borel measurable set, provided
that the integral on the right-hand side is well defined. Note that,
since $T$ is only locally Lipschitz, it may happen that $\nor
(X_1,\ldots ,X_q)$ has not locally finite ${\mathcal H}^{qd-1}$ measure.
Therefore, in order that the mixed curvature measures are well defined
as Radon measures, we shall assume that
\begin{equation} \label{bound}
\|C\|_{r_1,\ldots,r_q}(X_1,\ldots,X_q;\cdot)\mbox{ is locally finite
for all }0\leq r_1,\ldots ,r_q\leq d-1,
\end{equation}
where $\|C\|_{r_1,\ldots,r_q}(X_1,\ldots,X_q;\cdot)$ is the total
variation measure corresponding to \eqref{mixed-c-m}. 
In the case $q=2$, this condition has been considered in
\cite{RaZ02}; see also \cite{RaZ01}.  In Remark \ref{noremark} (1) below we explain why \eqref{bound} 
is satisfied whenever $X_1,\ldots,X_q$ are convex sets. Moreover, for 
sets $X_1,\ldots,X_q\subset\R^d$ of positive reach, it is proved in 
Proposition \ref{Prop3.4} that \eqref{bound} holds for $X_1,\rho_2 X_2,\ldots,\rho_q X_q$ 
for almost all rotations $\rho_2,\ldots,\rho_q$.

The mixed curvature measures are symmetric in the sense that for any
permutation $\sigma$ of $\{ 1,\ldots ,q\}$,
\begin{eqnarray*}
\lefteqn{C_{r_{\sigma (1)},\ldots ,r_{\sigma (q)}}(X_{\sigma (1)},\ldots
,X_{\sigma (q)};A_{\sigma (1)}\times\cdots\times A_{\sigma (q)}\times
B)}\hspace{5cm}\\
&=&C_{r_1,\ldots ,r_q}(X_1,\ldots ,X_q;A_1\times\cdots\times A_q\times
B)
\end{eqnarray*}
(see \cite[Proposition~1~(c)]{Rataj1997}). The definition of mixed
curvature measures is extended to arbitrary indices $0\leq r_i\leq d$ by setting
$$C_{d,\ldots ,d,r_{k+1},\ldots ,r_q}(X_1,\ldots ,X_q;\cdot )=
({\mathcal H}^d\meares X_1)\otimes\cdots\otimes ({\mathcal H}^d\meares X_k)\otimes
C_{r_{k+1},\ldots ,r_q}(X_{k+1},\ldots ,X_q;\cdot )$$
for $k=1,\ldots ,q-1$ and $r_{k+1},\ldots ,r_q\leq d-1$, provided that 
$r_{k+1}+\ldots +r_q\geq (q-k-1)d$, and by applying the symmetry.
Consequently, the mixed curvature measures are defined for all integers $0\leq
r_1,\ldots ,r_q\leq d$ with $(q-1)d\leq r_1+\cdots +r_q\leq qd-1$.

We say that the sets $X_1,\ldots ,X_q$ of positive reach {\it osculate} if
there exist $(x,u_i)\in\nor X_i$, $i=1,\ldots ,q$, such that
$o\in\cone\{ u_1,\ldots ,u_q\}$.

As already mentioned in the Introduction, the mixed curvature measures
appear in the translative integral formula for curvature measures of
intersections.

\begin{theorem}[{\cite[Theorem~1]{Rataj1997}}]  
Let $X_1,\ldots ,X_q$ be sets with positive reach in $\Rd$ (for $q\geq 2$)
which satisfy \eqref{bound} and such that
\begin{equation}   \label{NOC}
{\mathcal H}^{(q-1)d}(\{ (z_2,\ldots ,z_q): X_1,X_2+z_2,\ldots ,X_q+z_q\mbox{
osculate}\} )=0.
\end{equation}
Then, for any $k\in\{ 0,1,\ldots ,d-1\}$, the translative formula
\eqref{TIF} holds.
\end{theorem}
For conditions sufficient for \eqref{NOC}, see \cite{Rataj1997} and \cite{RaZ95}. In particular, 
\eqref{NOC} is satisfied if all sets are convex, or if all sets are sufficiently smooth, or for 
arbitrary sets with positive reach in case $d=2$. Moreover, if  $X_1,\ldots ,X_q$ are arbitrary sets with positive reach in $\Rd$,  
then  $X_1,\rho_2X_2, \ldots,\rho_qX_q$ satisfy \eqref{NOC} for  
almost all rotations $\rho_2,\ldots,\rho_q$ (see \cite[Remark 3.2]{RaZ01}). 

\section{An integral representation of mixed curvature measures}

In this section we derive a representation of mixed curvature measures
as integrals over the product of unit normal bundles of the sets
involved.

\medskip

Let $q\geq 2$ and let $X_1,\ldots ,X_q$ be sets with positive reach.
The principal curvatures and the principal directions of curvature of $X_j$ at $(x_j,u_j)\in\nor X_j$ 
will be denoted by $k_i^{(j)}(x_j,u_j)$ and $a_i^{(j)}(x_j,u_j)$, respectively, for $i=1,\ldots
,d-1$. In the following, we use the short notation 
$$\mbK^{(j)}(x_j,u_j):=\prod_{i=1}^{d-1}\sqrt{1+\big( k_i^{(j)}(x_j,u_j)\big)^2},$$
arguments of the curvature functions will often be omitted if these
will be clear from the context. 
We shall also shortly write $\underline{x}:=(x_1,\ldots,x_q)$ and 
$$\underline{(x,u)}:=(x_1,u_1,\ldots ,x_q,u_q)\in\underline{\NK}.$$
For $\underline{(x,u)}\in R^c$, $s\in \R^q$ and $t\in S^{q-1}_+$, we set
\[\tilde{u}(s):=\sum_{i=1}^qs_iu_i\qquad\text{and}\qquad 
\underline{u}(t):=\frac{\sum_{i=1}^qt_iu_i}{|\sum_{i=1}^qt_iu_i|}.\]
Further, given $\underline{r}=(r_1,\ldots ,r_q)$ with $0\leq r_1,\ldots
,r_q\leq d-1$, $r_1+\ldots +r_q\geq (q-1)d$ and a Borel set $A\subset\R^{qd}\times S^{d-1}$, we define
\[\mu_{\underline{r}}(\underline{(x,u)};A):=\frac{1}{\omega_{d-k}}
\int_{S^{q-1}_+}\I_A(\underline{x},\underline{u}(t))
\prod_{i=1}^qt_i^{d-r_i}\left|\tilde{u}(t)
\right|^{-(d-k)}\mathcal{H}^{q-1}(dt),\]
if $u_1,\ldots,u_q\in S^{d-1}$ are linearly independent (and hence $o\notin \text{cone}\{u_1,\ldots,u_q\}$), and otherwise 
we define 
$\mu_{\underline{r}}(\underline{(x,u)};A):=0$.

\medskip

Now we can state our main result.

\medskip

\begin{theorem}\label{MiCurv}
Let $q,d\ge 2$, let $X_1,\ldots,X_q\subset\Rd$ be sets with
positive reach satisfying \eqref{bound}, let
$r_1,\ldots,r_q\in\{0,\ldots,d-1\}$ with $r_1+\ldots+r_q\ge
(q-1)d$ and put $k:=r_1+\ldots+r_q-(q-1)d$. Further, let $A\subset\R^{qd}\times 
S^{d-1}$ be Borel measurable and bounded. Then
\begin{align} \label{MC}
C_{r_1,\ldots,r_q}(X_1,\ldots,X_q;A)
=&\int_{\underline{\NK}}\mu_{
\underline{r}}(\underline{(x,u)};A)\sum_{\substack{|I_j|=r_j\\ 
j=1,\ldots,q}}\prod_{j=1}^q\frac{\prod_{i\in I_j^c}k_i^{(j)}}{\mbK^{(j)}}
\nonumber \\
&\qquad\times\left|\bigwedge_{j=1}^q\bigwedge_{i\in I_j^c}a_i^{(j)}\wedge
u_1\wedge\ldots\wedge u_q\right|^2\mathcal{H}^{q(d-1)}(d\underline{(x,u)})\,.
\end{align}
\end{theorem}

We postpone the proof of this theorem to the next section and first discuss assumption \eqref{bound} and consider some special cases of Theorem \ref{MiCurv}. 

\begin{remark}\label{noremark} \rm
\begin{enumerate}
\item For convex bodies $K_1,\ldots ,K_q$, condition \eqref{bound} is
always satisfied for the following reason. All the mixed curvature
measures are nonnegative in this case. Therefore they are always
well defined, though possibly not finite. Nevertheless, the translative
formula \eqref{TIF} must be true in this case and since its left-hand
side is clearly locally bounded, the mixed curvature measures on
the right-hand side will be locally bounded as well (cf.\ \cite{RaZ02}).
\item For convex bodies $K_1,K_2\in\K^d$ and $\alpha\in\{1,\ldots,
d-1\}$, the relationship
\[\binom{d}{\alpha}V(K_1[\alpha],K_2[d-\alpha])=C_{\alpha,d-\alpha}
(K_1,-K_2;\R^{2d}\times S^{d-1})\]
is well-known.  A curvature based representation of general mixed volumes is provided 
in \cite{HugHab} and will be  developed further in future work. 
\item By definition and using the preceding notation, we have
\[\left|\bigwedge_{j=1}^q\bigwedge_{i\in I_j^c}a_i^{(j)}\wedge
u_1\wedge\ldots\wedge u_q\right|^2=\left[\text{lin}\{a_i^{(1)}:i\in I_1\},\ldots,
\text{lin}\{a_i^{(q)}:i\in I_q\}\right]^2\,,\]
where the bracket on the right-hand side was already defined in \eqref{star2}; see also 
the references after \eqref{star2} and \cite{WeilPreprint} or \cite[p.~598]{SW08}.
\end{enumerate}
\end{remark}

Note that if \eqref{bound} is not satisfied, then \eqref{MC} might not define a (signed) measure since $\mu_{\underline{r}}$ is not bounded in general. Theorem~\ref{MiCurv} still gives us at least an expression for the (not necessary locally bounded) total variation measures $\| C\|_{r_1,\ldots ,r_q}(X_1,\ldots ,X_q;\cdot )$,
with the integrand on the right-hand side of \eqref{MC} replaced by its 
absolute value. 

To prepare the proof of a condition which ensures that \eqref{bound} is satisfied, we first provide the bounds given in the next lemma. 

\begin{lemma}  \label{L_est}
Let the assumptions of Theorem~\ref{MiCurv} be satisfied with $A=B\times S^{d-1}$, for a Borel set $B\subset\R^{dq}$. 
Then, $\mu_{\underline{r}}(\underline{(x,u)};A)=0$ if $u_1, \ldots, u_q$ are linearly dependent or $\underline{x}\notin B$, and 
\begin{equation}   \label{E_mu}
c_{d,q,k}\;\mu_{\underline{r}}(\underline{(x,u)};A)\leq
\begin{cases}
(1+|\ln(|u_1\wedge\dots\wedge u_q|)|)|u_1\wedge\dots\wedge u_q|^{-1}&\text{\rm if } k=d-q,\\
|u_1\wedge\dots\wedge u_q|^{-(d-k-q+1)}&\text{\rm if } k<d-q,
\end{cases}
\end{equation}
with some constants $c_{d,q,k}$, otherwise. In particular, 
\begin{eqnarray}   
&&\| C\|_{r_1,\ldots ,r_q}(X_1,\ldots ,X_q; B\times S^{d-1})\nonumber\\
&&\qquad \leq
\operatorname{const}\int_{\underline{\NK}}\I_B(\underline{x})
|u_1\wedge\cdots\wedge u_q|^{-(d-q)}\, {\mathcal
H}^{q(d-1)}(d(\underline{(x,u)})).
\label{bound-suff}
\end{eqnarray}
\end{lemma}

\begin{proof}
Assume that $u_1, \ldots, u_q$ are linearly independent and $\underline{x}\in B$. 
Then, from the definition of $\mu_{\underline{r}}(\underline{(x,u)};A)$, we easily get
$$\mu_{\underline{r}}(\underline{(x,u)};A)\leq\frac{1}{\omega_{d-k}}\int_{S^{q-1}_+}|\tilde{u}(t)|^{-(d-k)}\, \mathcal{H}^{q-1}(dt).$$
For the given linearly independent vectors $\underline{u}=(u_1,\dots,u_q)$, let $\Delta_{\underline{u}}$ denote the 
convex hull of $u_1,\dots, u_q$, and 
let $w\in S^{q-1}$ be a unit vector in the linear hull of $u_1,\ldots,u_q$ and perpendicular to $\Delta_{\underline{u}}$. Consider the differentiable, one-to-one mapping
$$h: S^{q-1}_+\to\Delta_{\underline{u}},\quad t\mapsto (t_1+\dots+t_q)^{-1}\tilde{u}(t).$$
In order to compute the Jacobian $J_{q-1}h(t)$ of $h$ at $t\in S^{q-1}_+$, let $\{v_1,\dots,v_{q-1},t\}$ be an orthonormal basis of $\R^q$ and note that
$$Dh_t(v_j)=(t_1+\dots+t_q)^{-1}(v_j^1u_1+\dots+v_j^qu_q)+\alpha_j(t)\tilde{u}(t),\quad j=1,\dots,q-1,$$
for some $\alpha_j(t)\in\R$ and
$$\left|\bigwedge_{j=1}^{q-1}Dh_t(v_j)\wedge \tilde{u}(t)\right|=\frac{|u_1\wedge\dots\wedge u_q|}{(t_1+\dots+t_q)^{q-1}}.$$
Hence,
$$J_{q-1}h(t)=\left|\bigwedge_{i=1}^{q-1}Dh_t(v_j)\right|=\frac{|u_1\wedge\dots\wedge u_q|}{(t_1+\dots+t_q)^{q-1}|\tilde{u}(t)\cdot w|}
\geq \frac{|u_1\wedge\dots\wedge u_q|}{q^{(q-1)/2}|\tilde{u}(t)|},\quad t\in S^{q-1}_+.$$
Since clearly $|\tilde{u}(t)|\geq |h(t)|$, $t\in S^{q-1}_+$,  the area formula implies that
\begin{eqnarray*}
q^{-(q-1)/2}\omega_{d-k}|u_1\wedge\dots\wedge u_q|\mu_{\underline{r}}(\underline{(x,u)};S^{q-1}_+)
&\leq& \int_{\tilde{\Delta}_{\underline{u}}}|z|^{-(d-k-1)}\, \mathcal{H}^{q-1}(dz)\\
&\leq&\int_{B^{(q-1)}_1}(\rho^2+|x|^2)^{-\frac{d-k-1}2}\, \mathcal{H}^{q-1}(dx),
\end{eqnarray*}
where $\rho:=\operatorname{dist}(o,\Delta_{\underline{u}})\le 1$ and $B^{(q-1)}_1$ denotes the unit ball in $\R^{q-1}$ with centre at the origin. The last integral can be bounded from above by
\begin{eqnarray*}
\omega_{q-1}\int_0^1 r^{q-2}(\rho^2+r^2)^{-\frac{d-k-1}2}\, dr
&=&\omega_{q-1}\rho^{-(d-k-1)}\int_0^1 \frac{r^{q-2}}{(1+(\frac r\rho)^2)^{\frac{d-k-1}2}}\, dr\\
&=&\omega_{q-1}\rho^{-(d-k-q)}\int_0^{\rho^{-1}} \frac{s^{q-2}}{(1+s^2)^{\frac{d-k-1}2}}\, ds\\
&\le &\omega_{q-1}\rho^{-(d-k-q)}\left( 1+\int_1^{\rho^{-1}} s^{-(d-k-q+1)}\right)\, ds,
\end{eqnarray*}
and the last integral can be easily evaluated.

The proof of \eqref{E_mu} will be finished by the following estimate. The norm $|u_1\wedge\cdots\wedge u_q|$ is 
equal to the $q$-volume of the parallelepiped spanned by the vectors $u_1,\ldots,u_q$,  
and since these are unit vectors, we get
$$
|u_1\wedge\cdots\wedge u_q|\leq\kappa_{q-1}\frac 2q\rho,
$$
where $\kappa_n$ is the volume of the $n$-dimensional unit ball. 
Applying now Theorem~\ref{MiCurv} and the fact that 
$|\xi\wedge u_1\wedge\dots\wedge u_q|\leq |u_1\wedge\dots\wedge u_q|$ for any simple unit multivector $\xi$, 
we obtain \eqref{bound-suff}.
\end{proof}

Using Lemma \ref{L_est} we now show that mixed curvature measures are defined for generic rotations of 
sets with positive reach (cf.\ \cite[Proposition~4.6]{RaZ01}). Let $\nu_d$ denote the normalized
invariant measure on  the group  $\SOd$ of proper rotations of $\R^d$.

\begin{proposition}\label{Prop3.4}
Let $q,d\geq 2$ and let $X_1,\ldots ,X_q\subset\Rd$ be compact sets with
positive reach. Then \eqref{bound} is satisfied by $X_1,\rho_2
X_2,\ldots ,\rho_qX_q$ for $(\nu_d)^{q-1}$-almost all rotations
$\rho_2,\ldots ,\rho_q\in\SOd$.
\end{proposition}

\begin{proof}
Due to \eqref{bound-suff}, 
it is sufficient to show that
$$\int_{S^{d-1}}\cdots\int_{S^{d-1}}
|u_1\wedge\cdots\wedge u_q|^{-(d-q)} \,{\mathcal H}^{d-1}(du_q)\ldots {\mathcal
H}^{d-1}(du_2)<\infty$$
for any $u_1\in S^{d-1}$. First, observe that
$$|u_1\wedge\cdots\wedge u_q|=|u_1\wedge\cdots\wedge u_{q-1}|
|p_Vu_q|,$$
where $V$ is the orthogonal complement to the linear hull of
$u_1,\ldots ,u_{q-1}$ (note that $\dim V=d-q+1$) and $p_V$ denotes the orthogonal projection to
$V$. Moreover, a direct calculation shows that if $L$ is an $l$-dimensional linear subspace of $\R^d$,  $l\in\{1,\ldots,d-1\}$, 
and $p+l>0$, then
\begin{align}\label{directcalc}
&\int_{S^{d-1}}|p_Lu|^{p}\, {\mathcal H}^{d-1}(du)\nonumber\\
&\qquad= \int_{S^{d-1}\cap L}\int_{S^{d-1}\cap L^\perp}\int_0^{\frac{\pi}{2}} (\cos t)^{l-1}(\sin t)^{d-l-1}(\cos t)^p\, dt\, \mathcal{H}^{d-l-1}(dx)\, \mathcal{H}^{l-1}(dy) \nonumber\\
&\qquad=\omega_{l}\omega_{d-l} \frac{\Gamma\left(\frac{d-l}{2}\right)\Gamma\left(\frac{p+l}{2}\right)}{2\,\Gamma\left(\frac{d+p}{2}\right)}.
\end{align}
Applying \eqref{directcalc} with $L=V$, $l=d-q+1$ and $p=-d+q$ in a first step, it remains to be shown that 
$$\int_{S^{d-1}}\cdots\int_{S^{d-1}}
|u_1\wedge\cdots\wedge u_{q-1}|^{-(d-q)} \,{\mathcal H}^{d-1}(du_{q-1})\ldots {\mathcal
H}^{d-1}(du_2)<\infty.$$
Repetition of the preceding argument yields the assertion. 
\end{proof}
 
\medskip

Theorem \ref{MiCurv} can be specified in various ways. 
First, let $X_1,\ldots,X_q$ be sets with positive reach 
and $r_1=\cdots=r_q=d-1$ with $q,d\ge 2$. Then $k=d-q$ and 
\[\mu_{\underline{d-1}}(\underline{(x,u)};A)=\frac{1}{\omega_{q}}
\int_{S^{q-1}_+}\I_A(\underline{x},\underline{u}(t))\, 
t_1\cdots t_q \left|\tilde{u}(t)
\right|^{-q}\mathcal{H}^{q-1}(dt),\]
if $u_1,\ldots,u_q\in S^{d-1}$ are linearly independent, and zero otherwise. Furthermore,
\begin{align*}
C_{d-1,\ldots,d-1}(X_1,\ldots,X_q;A)&=2^q\int\ldots\int \mu_{\underline{d-1}}(\underline{(x,u)};A)\,|u_1\wedge\ldots\wedge u_q|^2\, \\
&\qquad\qquad 
\times \,C_{d-1}(X_q;d(x_q,u_q))\ldots C_{d-1}(X_1;d(x_1,u_1)).
\end{align*}

The special case where $X_1,\ldots,X_q$ are convex
polytopes but $r_1,\ldots,r_q\in\{0,\ldots,d-1\}$ are arbitrary, is of particular interest, 
since it shows that the representation of 
mixed curvature measures given in Theorem \ref{MiCurv} extends the defining 
relationship (3.1) in \cite{WeilPreprint} in a natural way. 

For a polytope $P\subset\R^d$ and $j\in\{0,\ldots,d-1\}$, 
we write $\mathcal{F}_j(P)$ for the set of all $j$-dimensional faces of $P$ (see \cite[p.~16]{Schneider}), and $N(P,F)$ for 
the normal cone of $P$ at a face $F$ of $P$ (see \cite[p.~83]{Schneider}). 
For faces $F_i\in \mathcal{F}_{r_i}(P_i)$,  $i=1,\ldots,q$, 
the bracket $\LK F_1,\ldots,F_q\RK$ is defined as in \cite{WeilPreprint} or \cite[p.~598]{SW08}.

\begin{corollary}\label{CorPoly}
Let $q,d\ge 2$, and let $P_1,\ldots,P_q$ be convex polytopes (or polyhedral sets).   
Let $r_1,\ldots,r_q\in\{0,\ldots,d-1\}$ 
with $r_1+\ldots+r_q\ge 
(q-1)d$ and $k:=r_1+\ldots+r_q-(q-1)d$. Further, let $B\subset\R^{qd}$ and 
$C\subset S^{d-1}$ be Borel measurable sets. Then
\begin{eqnarray*}
\lefteqn{C_{r_1,\ldots,r_q}(P_1,\ldots,P_q;B\times C)}\hspace{1cm}\\
&=&\sum_{F_1\in\mathcal{F}_{r_1}(P_1)}\ldots \sum_{F_q\in\mathcal{F}_{r_q}(P_q)}
\frac{\mathcal{H}^{d-1-k}\left(\left(\sum_{i=1}^q N(P_i,F_i)\right)\cap 
C\right)}{\omega_{d-k}}\\
&&\hspace{4cm}\times\,\LK F_1,\ldots,F_q\RK\left(
\otimes_{i=1}^q\left(\mathcal{H}^{r_i}\meares F_i\right)\right)(B)\,.
\end{eqnarray*}
\end{corollary}

\medskip

In particular, Corollary \ref{CorPoly} is an extension of the defining relation 
 in \cite[(3.1)]{WeilPreprint}, since 
\[\gamma(F_1,\ldots,F_q;P_1,\ldots,P_q)=\frac{
\mathcal{H}^{d-1-k}\left(\left(\sum_{i=1}^q N(P_i,F_i)\right)
\cap S^{d-1}\right)
}{\omega_{d-k}}
\,,\]
provided that $\text{lin } N(P_1,F_1),\ldots,\text{lin } N(P_q,F_q)$ are
linearly independent subspaces. Also note that if these subspaces are not linearly independent, 
then 
$$
\text{dim}\left(\left(\sum_{i=1}^q N(P_i,F_i)\right)
\cap S^{d-1}\right)<d-r_1+\ldots+d-r_q-1=d-k-1,
$$
and hence $\gamma(F_1,\ldots,F_q;P_1,\ldots,P_q)=0$ in this case.
\medskip

\begin{proof} We continue to use the previous notation. Under the present special 
assumptions, the formula of Theorem 
\ref{MiCurv} yields
\begin{align*}
C_{r_1,\ldots,r_q}(P_1,\ldots,&P_q;B\times C)\\
=&\sum_{F_1\in\mathcal{F}_{r_1}(P_1)}\ldots \sum_{F_q\in\mathcal{F}_{r_q}(P_q)}
\LK F_1,\ldots,F_q\RK\left(
\otimes_{i=1}^q\left(\mathcal{H}^{r_i}\meares F_i\right)\right)(B)\\
&\qquad\times \frac{1}{\omega_{d-k}}\int_{S^{q-1}_+}\int_{N(P_1,F_1)\cap S^{d-1}}\ldots
\int_{N(P_q,F_q)\cap S^{d-1}}\left(\prod_{j=1}^qt_j^{d-1-r_j}\right)\\
&\qquad\times
\left|\bigwedge_{j=1}^q\bigwedge_{i\in I_j^c}a_i^{(j)}\wedge
u_1\wedge\ldots\wedge u_q\right|\I_C(\underline{u}(t))
|\tilde{u}(t)|^{-(d-k)}\\
&\qquad\times\mathcal{H}^{d-1-r_q}(du_q)\ldots\mathcal{H}^{d-1-r_1}
(du_1)\,\mathcal{H}^{q-1}(dt)\,,
\end{align*}
where $\{a_i^{(j)}:i\in I_j^c\}$ is an orthonormal basis of
$\text{Tan}(N(P_j,F_j)\cap S^{d-1},u_j)$ and $\{a_i^{(j)}:i\in I_j\}$ spans
$\text{lin}(F_j-F_j)$, $j=1,\ldots,q$. Here we adopt the convention that 
the integrand is zero if $u_1,\ldots,u_q$ are linearly dependent. 

Let $F_j\in \mathcal{F}_{r_j}(P_j)$, 
$j=1,\ldots,q$, be fixed and assume that $\text{lin } 
N(P_1,F_1),\ldots,\text{lin } N(P_q,F_q)$ are linearly independent. Consider the bijective 
map
\begin{align*}
T:(N(P_1,F_1)\cap S^{d-1})\times\ldots\times (N(P_q,F_q)\cap S^{d-1})\times S^{q-1}_+&\to
\left(\sum_{i=1}^q N(P_i,F_i)\right)
\cap S^{d-1}\,,\\
(u_1,\ldots,u_q,t)&\mapsto \underline{u}(t)\,.
\end{align*}
Then the required equality for any such summand follows by an application of the area 
formula once we have checked that
\[J_{d-k-1}T(\underline{u},t)=|\tilde{u}(t)|^{-(d-k)}
\left(\prod_{j=1}^qt_j^{d-1-r_j}\right)
\left|\bigwedge_{j=1}^q\bigwedge_{i\in I_j^c}
a_i^{(j)}\wedge
u_1\wedge\ldots\wedge u_q\right|\,.\]
In fact, using the previous notation, we find that
\[\frac{\partial T}{\partial 
a_i^{(j)}}(\underline{u},t)=\frac{t_ja_i^{(j)}}{|\tilde{u}(t)|}+\lambda_i^{(j)}\underline
{u}(t)\,,\]
where $i\in I_j^c$, $j\in\{1,\ldots,q\}$, and $\lambda_i^{(j)}\in\R$,
\[\frac{\partial T}{\partial f_l}(\underline{u},t)=\frac{\tilde{u}(f_l)}{
|\tilde{u}(t)|}+\lambda_l\underline{u}(t)\,,\]
where $l\in \{1,\ldots,q-1\}$ and $\lambda_l\in\R$, and
\[\left\langle\frac{\partial T}{\partial 
a_i^{(j)}}(\underline{u},t),\underline{u}(t)\right\rangle=\left\langle\frac{
\partial T}{\partial f_l}(\underline{u},t),\underline{u}(t)\right\rangle=0\,.\]
Here $f_1,\ldots,f_{q-1},t$ is an orthonormal basis of $\R^q$. Thus 
\begin{align*}
J_{d-k-1}T(\underline{u},t)&=\left|\bigwedge_{j=1}^q\bigwedge_{i\in I_j^c}
\left(\frac{t_ja_i^{(j)}}{|\tilde{u}(t)|}+\lambda_i^{(j)}\underline{u}(t)\right)\wedge
\bigwedge_{i=1}^{q-1}\left(\frac{\tilde{u}(f_i)}{
|\tilde{u}(t)|}+\lambda_i\underline{u}(t)\right)\right|\\
&=\left|\bigwedge_{j=1}^q\bigwedge_{i\in I_j^c}
\left(\frac{t_ja_i^{(j)}}{|\tilde{u}(t)|}\right)\wedge
\bigwedge_{i=1}^{q-1}\left(\frac{\tilde{u}(f_i)}{
|\tilde{u}(t)|}\right)\wedge\underline{u}(t)\right|\\
&=|\tilde{u}(t)|^{-(d-k)}\left|\bigwedge_{j=1}^q\bigwedge_{i\in I_j^c}
(t_ja_i^{(j)})\wedge\bigwedge_{i=1}^{q-1}\tilde{u}(f_i)\wedge\tilde{u}(t)\right|\,,
\end{align*}
from which the formula for the Jacobian immediately follows.

If $\text{lin } N(P_1,F_1),\ldots,\text{lin } N(P_q,F_q)$ are not linearly independent, then 
$\LK F_1,\ldots,F_q\RK=0$, and thus the requested equality is also true in this case.

\end{proof}

Further representation formulas,  which are needed for the analysis of Boolean models  
in stochastic geometry can be derived from Theorem \ref{MiCurv} and Corollary \ref{CorPoly}. 
Various examples of such results and their applications are provided in \cite{GW02}, \cite[Section 3]{HoermannDiss} and \cite{HHKM}.

\section{Proof of Theorem~\ref{MiCurv}}

For given $t\in S^{q-1}_+$ we denote by $f_1,\ldots,f_{q-1},t$ an
orthonormal basis of $\R^q$ whose orientation is chosen in such a way that
\[\det(f_1,\ldots,f_{q-1},t)=(-1)^{(d-1)\binom{q}{2}}\,.\]
Denote
\begin{eqnarray*}
\tilde{a}(\underline{(x,u)},t)&=&
\left(\wedge_{d-1}\Pi_1\right)a_{X_1}(x_1,u_1)\wedge\cdots\wedge
\left(\wedge_{d-1}\Pi_q\right)a_{X_q}(x_q,u_q)\\ 
&&\hspace{3cm}\wedge\left(\wedge_{q-1}\Pi_{q+1}\right)(f_1\wedge\cdots\wedge f_{q-1}),
\end{eqnarray*}
where $\Pi_i$ are the canonical embeddings into $(\R^{2d})^q\times\R^{q-1}$ such that
$$(a_1,\ldots,a_q,a_{q+1})=\sum_{i=1}^{q+1}\Pi_ia_i;$$
note that $\tilde{a}(\underline{(x,u)},t)$ is a $(qd-1)$ vector
field tangent to $\underline{\NK}\times S^{q-1}$. Using the area
formula for currents \cite[\S4.1.30]{F69} and Lemma~\ref{LemmaRataj},
we obtain
\begin{eqnarray*}
&&C_{r_1,\ldots,r_q}(X_1,\ldots,X_q;A)\\
&&\qquad=\int_{\underline{\NK}}\int_{S^{q-1}_+}
\left\langle\wedge_{qd-1}\text{ap
}DT(\underline{(x,u)},t)\tilde{a}(\underline{(x,u)},t),\varphi_{r_1,\ldots,r_q}
(\underline{u}(t))\right\rangle\\
&&\hspace{4cm}\times\,\I_A(\underline{x},\underline{u}(t))\,\mathcal{H}^{q-1}(dt)\,
\mathcal{H}^{q(d-1)}(d\underline{(x,u)})\,,
\end{eqnarray*}
provided that the orientation is chosen properly, i.e., that
$\wedge_{qd-1}\text{ap}DT(\underline{(x,u)},t)\tilde{a}(\underline{(x,u)},t)$
is a positive multiple of $a_{X_1,\ldots ,X_q}$; this will be verified
later.

A direct calculation shows that, for $\mathcal{H}^{qd-1}$-almost all 
$(\underline{(x,u)},t)\in\underline{\NK}\times S^{q-1}_+$,
\begin{align*}
\wedge_{qd-1}\text{ap }DT \tilde{a}(\underline{(x,u)},t)
=&\phantom{\wedge}\frac{1}{\mbK_1}\bigwedge_{i=1}^{d-1}\left(a_i^{(1)},o,
\ldots,o,\frac{t_1k_i^{(1)}}{|\tilde{u}(t)|}a_i^{(1)}+\lambda_i^{(1)}
\tilde{u}(t)
\right)\wedge
\\
&\quad\vdots\\
&\wedge\frac{1}{\mbK_q}\bigwedge_{i=1}^{d-1}\left(o,\ldots,o,a_i^{(q)},
\frac{t_qk_i^{(q)}}{|\tilde{u}(t)|}a_i^{(q)}+\lambda_i^{(q)}\tilde{u}
(t)\right)\\
&\wedge\bigwedge_{j=1}^{q-1}\left(o,\ldots,o,\frac{\tilde{u}(f_j)}{
|\tilde{u}(t)|}+\lambda_j\tilde{u}(t)\right)\,,
\end{align*}
where $\lambda_i^{(j)}$, $i\in\{1,\ldots,d-1\}$ and $j\in\{1,\ldots,q\}$, and
$\lambda_j$, $j\in\{1,\ldots,q-1\}$, are suitably chosen. We write
$\text{Sh}^*(r_1,\ldots,r_{q+1})$ for the set of all $\sigma\in 
\text{Sh}(r_1,\ldots,r_{q+1})$ which satisfy
\begin{eqnarray*}
\sigma(\{1,\ldots,R_1\})&\subset&\{1,\ldots,d-1\}\\
&\vdots&\\
\sigma(\{R_{q-1}+1,\ldots,R_q\})&\subset&\{(q-1)(d-1)+1,\ldots,q(d-1)\}\,,
\end{eqnarray*}
and then we define
$$I_\sigma(i)=\sigma(\{R_{i-1}+1,\ldots,R_i\}) -(i-1)(d-1)$$
and
$$I_\sigma(i)^c=\{ 1,\ldots ,d-1\}\setminus I_\sigma (i),\quad
i=1,\ldots ,q,$$
for $\sigma\in\text{Sh}^*(r_1,\ldots,r_{q+1})$. By $I_\sigma
(j)I_\sigma (j)^c$ we shall denote the permutation of $\{ 1,\ldots
,d-1\}$ mapping the first $r=|I_\sigma (j)|$ elements increasingly on
$I_\sigma (j)$ and the remaining $d-1-r$ elements increasingly  on
$I_\sigma (j)^c$.
Thus we arrive at
\begin{align}
&\left\langle\wedge_{qd-1}\text{ap }DT 
\tilde{a}(\underline{(x,u)},t),\varphi_{r_1,\ldots,r_q}(\underline{u}(t))
\right\rangle
=\frac{1}{\omega_{d-k}}(-1)^{c_1(d,r_1,\ldots,r_q)}  \nonumber\\
&\qquad \times\sum_{\sigma\in\text{Sh}^*(r_1,\ldots,r_{q+1})}\text{sgn}(\sigma)\left(\prod_{
j=
1}^q t_j^{d-1-r_j}\right)\prod_{j=1}^q\frac{\prod_{i\in 
I_\sigma(j)^c}k_i^{(j)}}{\mbK_j}|\tilde{u}(t)|^{-(d-k)}\nonumber\\
&\qquad\times\left[\bigwedge_{i\in I_\sigma(1)}a_i^{(1)},\ldots,\bigwedge_{i\in
I_\sigma(q)}a_i^{(q)},\bigwedge_{i\in
I_\sigma(1)^c}a_i^{(1)}\wedge\ldots\wedge\bigwedge_{i\in I_\sigma(q)^c}a_i^{(q)}
\wedge\bigwedge_{i=1}^{q-1}\tilde{u}(f_i)\wedge\tilde{u}(t)\right]\,.
\label{A4.5}
\end{align}
Observe that
\begin{equation}
\bigwedge_{i=1}^{q-1}\tilde{u}(f_i)\wedge\tilde{u}(t)=
\det(f_1,\ldots,f_{q-1},t)\,\,u_1\wedge\ldots\wedge u_q,
\end{equation}
\begin{equation}
\text{sgn}(\sigma)=\left(\prod_{j=1}^q\text{sgn}(I_\sigma(j)I_\sigma(j)^c)
\right)
(-1)^{c_2(d,r_1,\ldots,r_q)}
\end{equation}
with
\begin{eqnarray*}
&&c_2(d,r_1,\ldots ,r_q)\\
&&\qquad=\sum_{j=1}^q (d-1-r_j)(R_q-R_j)\\
&&\qquad\sim q(d-1)R_q+(d-1)\sum_{i=1}^qR_i+R_q\sum_{i=1}^qr_i+\sum_{i=1}^qr_iR_i\\
&&\qquad\sim q(d-1)R_q+(d-1)\left((q+1)\sum_{i=1}^q r_i-\sum_{i=1}^q ir_i\right)+R_q+\sum_{i=1}^qr_i+\sum_{1\le i<j\le q}r_ir_j\\
&&\qquad\sim (d-1)\sum_{i=1}^q r_i+(d-1)\sum_{i=1}^q ir_i+\sum_{1\le i<j\le q}r_ir_j,
\end{eqnarray*}
$$*\left(\bigwedge_{i\in I_\sigma (j)}a_i^{(j)}\right) =
\text{sgn}(I_\sigma(j)I_\sigma(j)^c)\ \bigwedge_{i\in I_\sigma (j)^c}a_i^{(j)}
\wedge u_j$$
and
$$\bigwedge_{j=1}^q(\bigwedge_{i\in I_\sigma (j)^c}a_i^{(j)}\wedge u_j)
=(-1)^{c_3}\bigwedge_{j=1}^q\bigwedge_{i\in I_\sigma(q)^c}a_i^{(j)}
\wedge\bigwedge_{i=1}^qu_i$$
with
$$c_3=\sum_{j=2}^q(j-1)(d-1-r_j)\sim (d-1)\binom{q}{2}+\sum_iir_i+\sum_ir_i.$$
Using \eqref{star2}, we thus have
\begin{align}
&\text{sgn}\,\sigma\left[\bigwedge_{i\in I_\sigma(1)}a_i^{(1)},\ldots,\bigwedge_{i\in
I_\sigma(q)}a_i^{(q)},\bigwedge_{i\in
I_\sigma(1)^c}a_i^{(1)}\wedge\ldots\wedge\bigwedge_{i\in I_\sigma(q)^c}a_i^{(q)}
\wedge\bigwedge_{i=1}^{q-1}\tilde{u}(f_i)\wedge\tilde{u}(t)\right]
\nonumber\\
&\quad=(-1)^{c_2}\left\langle\bigwedge_{j=1}^q(\bigwedge_{i\in I_\sigma
(j)^c}a_i^{(j)}\wedge u_j)\wedge
*\left(\bigwedge_{j=1}^q\bigwedge_{i\in I_\sigma(j)^c}a_i^{(j)}
\wedge\bigwedge_{i=1}^{q-1}\tilde{u}(f_i)\wedge\tilde{u}(t)\right)
,\Omega^d\right\rangle\nonumber\\
&\quad=(-1)^{c_2+c_3+(d-1)\binom{q}{2}}
\left\langle\bigwedge_{j=1}^q\bigwedge_{i\in I_\sigma (j)^c}a_i^{(j)}\wedge
\bigwedge_{j=1}^qu_j\wedge
*\left(\bigwedge_{j=1}^q\bigwedge_{i\in I_\sigma(q)^c}a_i^{(j)}
\wedge\bigwedge_{i=1}^qu_j\right) ,\Omega^d\right\rangle\nonumber\\
&\quad=(-1)^{c_2+c_3+(d-1)\binom{q}{2}}
\left|\bigwedge_{j=1}^q\bigwedge_{i\in I_\sigma(j)^c}a_i^{(j)}
\wedge\bigwedge_{j=1}^{q}u_j\right|,
\label{D4.5}
\end{align}
where we have used (\ref{star}) in the last step.
Combining \eqref{A4.5} and \eqref{D4.5} and observing that
$$c_1+c_2+c_3+(d-1)\binom{q}{2}\sim 0,$$ 
we finally obtain the required formula.

It remains to verify that the orientation of the joint unit normal
bundle has been chosen appropriately, i.e., that
\begin{equation}   \label{sign}
\left\langle\wedge_{qd-1}\text{ap}DT(\underline{(x,u)},t)\tilde{a}(\underline{(x,u)},t),
\sum_{\substack{0\leq r_1,\ldots ,r_q\leq
d\\ r_1+\cdots +r_q\geq (q-1)d}}\varepsilon^{qd-1-r_1-\cdots -r_q}
\varphi_{r_1,\ldots ,r_q}(\underline{u}(t))\right\rangle >0
\end{equation}
for sufficiently small $\varepsilon >0$. Consider first the case when
$X_1,\ldots ,X_q$ have $C^{1,1}$ smooth boundaries. Since all
curvatures are finite in this case, we get from \eqref{A4.5} and
\eqref{D4.5} that
$$\langle\wedge_{qd-1}\text{ap}DT(\underline{(x,u)},t)\tilde{a}(\underline{(x,u)},t),
\varphi_{d-1,\ldots ,d-1} (\underline{u}(t))\rangle >0.$$
But since $\varphi_{r_1,\ldots ,r_q}$ vanishes over $\nor(X_1,\ldots
,X_q)$ if $r_j=d$ for some $j\in\{1,\ldots,q\}$, this is the leading term in
the polynomial expression of \eqref{sign} which therefore will be 
positive for small $\varepsilon>0$. General sets $X_1,\ldots ,X_q$ of
positive reach can be approximated by parallel bodies with $C^{1,1}$
smooth boundaries so that the corresponding unit normal cycles are
arbitrarily close in the flat norm (see \cite{RaZ01}). Thus, the
expression in \eqref{sign} can be approximated by the corresponding one
for the parallel bodies and since it can never be  zero for
sufficiently small $\varepsilon>0$, it will remain positive.


\begin{thebibliography}{99}

\bibitem{F59}
H. Federer: Curvature measures.
\newblock{\em Trans. Amer. Math. Soc.} 93 (1959), 418--491.

\bibitem{F69}
H. Federer: {\em Geometric Measure Theory}.
\newblock{Springer Verlag, Berlin}, 1969.

\bibitem{GW02}
P. Goodey and W. Weil: Representations of mixed measures with applications
to Boolean models. {\em Rend. Circ. Mat. Palermo, Ser. II Supplemento}, 
70 (2002), 325--346.

\bibitem{HoermannDiss}
J. H\"orrmann: The method of densities for non-isotropic Boolean models. PhD thesis, Karlsruhe Institute of Technology (KIT), KIT Scientific Publishing, 2015.

\bibitem{HHKM}
 J. H\"orrmann, D. Hug, M. Klatt and K. Mecke. Minkowski tensor density 
formulas for Boolean models. {\em Adv. Appl. Math.} 55 (2014), 48--85.

\bibitem{HugHab}
D. Hug: Measures, curvatures and currents in convex geometry. Habilitation thesis, Albert-Ludwigs-Universit\"at Freiburg, 
1999. 

\bibitem{Rataj1997}
J.~Rataj: The iterated version of a translative integral formula for sets of
 positive reach.
\newblock {\em Rend. Circ. Mat. Palermo (2) Suppl.} 46 (1997), 129--138.

\bibitem{RaZ95}
J.~Rataj and M.~Z{\"a}hle: Mixed curvature measures for sets of positive reach and a
  translative integral formula.
\newblock {\em Geom. Dedicata} 57 (1995), 259--283.

\bibitem{RaZ01}
J.~Rataj and M.~Z{\"a}hle: Curvatures and Currents for Unions of Sets with Positive
Reach, II.
\newblock {\em Ann. Global Anal. Geom.} 20 (2001), 1--21.

\bibitem{RaZ02}
J.~Rataj and M.~Z{\"a}hle: A remark on mixed curvature measures for sets with positive
reach.
\newblock{\em Beitr\"{a}ge Algebra Geom.} 43 (2002), 171--179.

\bibitem{Schneider}
R.~Schneider: {\em Convex Bodies: The Brunn-Minkowski Theory.} 2nd edn, Encyclopedia of
Mathematics and Its Applications {\bf 151}, Cambridge University Press, Cambridge, 2014.

\bibitem{SchneiderandWeil1986}
R.~Schneider and W.~Weil: Translative and kinematic integral formulae for curvature 
measures.
\newblock{\em Math. Nachr.} 129 (1986), 67--80.

\bibitem{SW08} 
R.~Schneider and W.~Weil: {\em Stochastic and Integral Geometry}. Springer, Berlin (2008)

\bibitem{Weil1990}
W.~Weil: Iterations of translative integral formulae and non-isotropic Poisson processes of particles.
\newblock {\em Math. Z.} 205 (1990), 531--549.

\bibitem{WeilPreprint}
W.~Weil: Mixed measures and functionals of translative integral geometry.
{\it Math.\ Nachr.} 223 (2001), 161--184.

\bibitem{Z86}
M.~Z{\"a}hle: Integral and current representation of Federer's curvature
measures.
\newblock{\em Arch. Math.} 46 (1986), 557--567.

\bibitem{Zaehle1999}
M.~Z{\"a}hle: Nonosculating sets of positive reach.
\newblock {\em Geom. Dedicata} 76 (1999), 183--187.

\end{thebibliography}
\end{document}